%% file: main.tex
\documentclass[a4paper]{article}

\usepackage[english]{babel}
\usepackage[utf8x]{inputenc}

\usepackage{booktabs}
\usepackage{tabu}
\usepackage[T1]{fontenc}

\usepackage[a4paper,top=3cm,bottom=2cm,left=3cm,right=3cm,marginparwidth=1.75cm]{geometry}

\usepackage{amsmath}
\usepackage{graphicx}
\usepackage[colorinlistoftodos]{todonotes}
\usepackage[colorlinks=true, allcolors=blue]
{hyperref}
\usepackage{amssymb}
\usepackage{amsthm}
\numberwithin{equation}{section}

\newtheorem*{theorem*}{Theorem}
\newtheorem{thm}{Theorem}
\newtheorem*{ack}{Acknowledgments}
\newtheorem{lem}{Lemma}
\newtheorem*{cor}{Corollary}

\title{\textbf{Notes}}
\title{On the best possible exponent for the error term in the lattice point counting problem on the first Heisenberg group }
\author{Yoav A. Gath\\ Faculty of Mathematics \\Technion, Israel Institute Of Technology}

\begin{document}
\maketitle

\input{paper}

\end{document}

%% file: paper.tex
\begin{abstract}
We use classical methods from analytic number theory to resolve the lattice point counting problem on the first Heisenberg group, in the case where the gauge function is taken to be the Cygan-Kor$\acute{a}$nyi Heisenberg-norm $\mathcal{N}_{4,1}(z,w)=(|z|^{4}+w^{2})^{1/4}$. In this case, our main theorem establishes the estimate $\mathcal{E}(x)=\Omega_{\pm}(x^{\frac{1}{2}})$, where $\mathcal{E}(x)=\mathcal{S}(x)-\frac{\pi^{2}}{2}x$ is the error term arising in the lattice point counting problem, $\mathcal{S}(x)$ is given by
$$\mathcal{S}(x)=\sum_{0\leq m^2+n^2<\, x}r_2(m)$$
and $r_2(m)=|\{a,b\in\mathbb{Z}:\,a^{2}+b^{2}=m\}|$ is the familiar sum of squares function. As a corollary, we deduce that the exponent $\frac{1}{2}$ in the upper-bound $\left|\mathcal{E}(x)\right|\ll x^{\frac{1}{2}}\log{x}$ obtained by Garg, Nevo \& Taylor can not be improved and is thus best possible, thereby resolving the lattice point counting problem for the case in hand.
\end{abstract}
\maketitle
 
\tableofcontents
\section{Introduction, notation and statement of results}
\subsection{General counting principles}
\textit{\textbf{Counting lattice points in Euclidean space}}. Given a region $\mathcal{B}\subset \mathbb{R}^{n}$, the problem of estimating the number of integral points in large homothetic regions $|\mathbb{Z}^{n}\cap t\mathcal{B}|$ has a long history which can be traced back to Gauss and Dirichlet. Under suitable assumptions on $\mathcal{B}$, we expect the quantity $|\mathbb{Z}^{n}\cap t\mathcal{B}|$ to be well approximated by $\lambda(t\mathcal{B})$ $=$ vol(t$\mathcal{B}$), and we measure the resulting discrepancy arising from this approximation by analyzing the error term $\mathcal{E}(t)=|\mathbb{Z}^{n}\cap t\mathcal{B}|-\lambda(t\mathcal{B})$. To get a handle on the error term, we introduce the notion of an admissible exponent. Say $\alpha$ is an admissible exponent if one has the estimate $\big|\mathcal{E}(t)\big|\ll_{\epsilon}(\lambda(t\mathcal{B}))^{\alpha+\epsilon}$ for any $\epsilon>0$, as $t\to\infty$. Setting $\kappa=\inf_{\alpha}$, where the infimum is taken over all admissible exponents, the (Euclidean) lattice point counting problem is that of establishing an asymptotic estimate of the form
\begin{equation}\label{eq:1.1}
|\mathbb{Z}^{n}\cap t\mathcal{B}|=\lambda(t\mathcal{B})+\mathcal{O}_{\epsilon}(\lambda(t\mathcal{B}))^{\kappa+\epsilon}
\end{equation}
For any $\epsilon>0$.\\\\
Below we list three motivating examples which have played an important role in the development of this subject.\\
$\bullet$ Euclidean balls. In $\mathbb{R}^{2}$ one has $\kappa\leq \frac{131}{416}$ due to Huxley \cite{huxley2003exponential}, while the conjectured value is $\kappa=\frac{1}{4}$. In $\mathbb{R}^{3}$ one has $\kappa\leq \frac{21}{48}$ due to Heath-Brown \cite{heath1999lattice}, while the conjectured value is $\kappa=\frac{1}{3}$. For $n\geq 4$, the value of $\kappa$ has been obtained, and its value is $\kappa=1-\frac{2}{n}$. For more details we refer the reader to \cite{atzel1988lattice}.\\
$\bullet$ Euclidean dilates of smooth compact convex bodies in $\mathbb{R}^{n}$ whose boundary surface has everywhere non-vanishing Gaussian curvature. These include amongst others, ellipsoids and other bodies of revolution, and for the relevant results on the error estimates we refer the reader to \cite{hlawka1950integrale}, \cite{herz1962number}, \cite{chamizo1998lattice} and \cite{ivic2006lattice}.\\
$\bullet$ Euclidean dilates of bodies in $\mathbb{R}^{n}$ whose boundary surface contains points of vanishing Gaussian curvature. This is a difficult case, and in general very little is known, except for a handful of special cases such as the unit balls of $l^{p}$-norm and some generalizations. For these cases, the effect of vanishing curvature on the error estimates has been extensively investigated, and we refer the reader to \cite{kratzel1999lattice}, \cite{kratzel2002lattice1}, and \cite{kratzel2002lattice2}.\\\\
The lattice point counting problem that we shall consider in the present paper will take a similar form to that of (1.1), the crucial difference being the dilation used in the expansion process. Our interest in this change of setting is motivated by viewing our lattice point counting problem in a wider context, as we now explain.\\\\
\textit{\textbf{Counting points in lattice subgroups}}.
In view of \eqref{eq:1.1}, it is natural to consider the following considerably more general set-up. Suppose $\mathbf{G}\subset \mathbf{GL}_n$ is a connected linear algebraic group defined over $\mathbb{Q}$ such that $\Gamma=\mathbf{G}(\mathbb{Z})$, its collection of integral points, forms a lattice subgroup in the group $\textit{G}=\mathbf{G}(\mathbb{R})\subset \mathbf{GL}_n(\mathbb{R})$ of real points. Let $\mathcal{B}_{t}\subset\textit{G}$ be a family of
bounded Borel subsets parameterized by $t\in\mathbb{R}_{+}$ satisfying $\mu(\mathcal{B}_t)\to\infty$, where $\mu$ denotes the Haar measure on $\textit{G}$ normalized so to have measure $1$ on a fundamental domain of $\Gamma$ in $\textit{G}$. In analogy with the Euclidean case, one aims to establish an asymptotic estimate of the form 
\begin{equation}\label{eq:1.2}
|\Gamma\cap\mathcal{B}_t|=\mu(\mathcal{B}_t)+\mathcal{O}_{\epsilon}(\mu(\mathcal{B}_t))^{\kappa+\epsilon}
\end{equation}
for any $\epsilon>0$,  where $\kappa$ is taken to be the infimum over all admissible exponents.
\\\\There are numerous methods to construct interesting families $\mathcal{B}_{t}\subset \textit{G}$. One such method, which is the one we shall adopt, is to consider a gauge $\omega$ defined on $\textit{G}$, i.e a left-invariant pseudo-distance function, and then take $\mathcal{B}_{t}=\{g\in \textit{G}:\omega(g,e)<t\}$. Of particular interest is the case where the gauge is taken to be a (proper) left-invariant distance function. In this direction, let us mention that \eqref{eq:1.2} has never been established, even in a single case, for any left-invariant distance function on any non-compact simple Lie group. For more details, as well as for other interesting examples and related results on the error terms, we refer the reader to \cite{gorodnik2012counting}.\\\\
In the present paper we intend to establish \eqref{eq:1.2} for the case of the first Heisenberg group, where the gauge is taken to be the Cygan-Kor$\acute{a}$nyi Heisenberg-norm. In fact, we shall first state the lattice point counting problem on an arbitrary Heisenberg group with respect to a certain family of gauges. This family, consisting of the so called Heisenberg-norms, arises naturally through the action of the dilation and unitary groups, and was considered in \cite{garg2015lattice} where the reader may find an in-depth and broad treatment of the lattice point counting problem on the Heisenberg groups.\\\\
We now turn to describe the set-up for the lattice point counting problem considered in \cite{garg2015lattice}, and we shall do so in the context of \eqref{eq:1.2}. Thus, for the time being we consider an Heisenberg group of arbitrary dimension along with any Heisenberg norm defined on it.\\\\
\begin{ack}
I would like to express my deepest gratitude to my advisor Prof. Amos Nevo for his support and guidance throughout the writing of this paper.   
\end{ack}
\subsection{Counting lattice points on the Heisenberg group}
\textit{\textbf{The Heisenberg group}}. The $q$-th Heisenberg group $\mathcal{H}_{q}$ has several equivalent descriptions, one of which is given by
\begin{equation}
\mathcal{H}_{q}=\mathbb{R}^{q}\times\mathbb{R}^{q}\times\mathbb{R}=\{(x,y,w):\,x,y\in\mathbb{R}^{q},\,w\in\ \mathbb{R} \}
\end{equation}
where the group law is defined by $$(x,y,w)(x',y',w')=(x+x',y+y',w+w'+\langle x,y'\rangle)\,,$$
and $\langle\cdotp , \cdotp\rangle$ denotes the standard inner product on $\mathbb{R}^{q}$.\\\\
An equivalent realization of $\mathcal{H}_{q}$ is given by the isomorphic group
\begin{equation}
\mathcal{H}_{q}=\mathbb{C}^{q}\times\mathbb{R}=\{(z,w):\,z=x+iy\in \mathbb{C}^{q},\,w\in\ \mathbb{R}\}
\end{equation}
where the group law is defined by $$(z,w)(z',w')=(z+z',w+w'+2\Im(z\cdotp z'))$$
so that multiplication can also be described by the symplectic form:
$$(x,y,w)(x',y',w')=(x+x',y+y',w+w'+2(\langle x',y\rangle-\langle x,y'\rangle))\,.$$
Having defined $\mathcal{H}_{q}$, we proceed by introducing two important groups of automorphisms acting on it.  The first is the dilation group, which is parameterized by 
$$\mathcal{H}_{q}\ni(z,w)\mapsto\phi_{a}(z,w)=(az,a^{2}w)\,\,, \,\,a\in\mathbb{R}_{+}\,.$$
The second group of automorphisms is given by the action of the unitary group $\mathbf{U}_{q}(\mathbb{C})$, where 
$$\mathcal{H}_{q}\ni(z,w)\mapsto\Phi_{\mathcal{U}}(z,w)=(\mathcal{U}z,w)\,\,,\,\,\mathcal{U}\in\mathbf{U}_{q}(\mathbb{C})\,.$$
\textit{\textbf{Heisenberg-norms}}. The action of the dilation group gives rise to the natural notion of homogeneity. Say $\psi:\mathcal{H}_{q}\rightarrow\mathbb{C}$ is homogeneous of degree $d$ if $$\psi\circ\phi_{a}=a^{d}\psi,\,\,\text{for any}\,\, a\in\mathbb{R}_{+}\,.$$
Likewise, the action of the unitary group gives rise to the natural notion of radiality, where we say that $\Psi:\mathcal{H}_{q}\rightarrow\mathbb{C}$ is radial if $$\Psi\circ\Phi_{\mathcal{U}}=\Psi ,\,\,\text{for any}\,\,\mathcal{U}\in\mathbf{U}_{q}(\mathbb{C})\,.$$
Having the notion of homogeneity and radiality, we now turn to the definition of the Heisenberg-norms. These norms are parameterized by $\alpha,A>0$, and take the form
\begin{equation}
\mathcal{N}_{\alpha,A}(z,w)=(|z|^{\alpha}+A|w|^{\alpha/2})^{1/\alpha}\,.
\end{equation}
It is evident that the Heisenberg-norms $\mathcal{N}_{\alpha,A}$ are radial and homogeneous of degree $1$. Of particular importance is the norm $\mathcal{N}_{4,1}$, the so called Cygan-Kor$\acute{a}$nyi Heisenberg-norm. This norm was considered by Cygan \cite{cygan1981subadditivity},  \cite{cygan1978wiener} and Kor$\acute{a}$nyi \cite{koranyi1985geometric}. Amongst its many properties, this gauge satisfies the inequality $\mathcal{N}_{4,1}(u\cdotp v)\leq \mathcal{N}_{4,1}(u)+\mathcal{N}_{4,1}(v)$ for all $u,v\in\mathcal{H}_{q}$, making the map $\delta(u,v)=\mathcal{N}_{4,1}(v^{-1}\cdotp u)$ a left-invariant distance function. One may also find this norm appearing in the context of harmonic analysis on the Heisenberg group. For example, it appears in the expression defining the fundamental solution of a natural sublaplacian on the Heisenberg group and in other natural kernels, and we refer the reader to \cite{stein1999harmonic} and  \cite{cowling2010unitary}. For us, the significance of the Cygan-Kor$\acute{a}$nyi Heisenberg-norm will be apparent through the use of Lemma 1 in the proof of the main theorem. \\\\ 
\textit{\textbf{Counting lattice points inside Heisenberg-dilated bodies}}. Having the relevant notation and definitions at hand, we now proceed to describe the lattice point counting problem considered in \cite{garg2015lattice}. Since $\mathcal{H}_{q}$ is parameterized by $\mathbb{R}^{2q+1}$, its lattice subgroup of integral points is parameterized by $\mathbb{Z}^{2q+1}$, and for the Haar measure on $\mathcal{H}_{q}$ we take the Lebesgue measure $\lambda=\lambda_{2q+1}$ on $\mathbb{R}^{2q+1}$. Next, for $\alpha,A>0$ we write $\mathcal{B}^{\alpha,A}_{t}\subset\mathbb{R}^{2q+1}$ for the $\mathcal{N}_{\alpha,A}$ ball of radius $t$ in $\mathcal{H}_{q}$, which is also the Heisenberg dilate by $t$ of the unit ball 
$$\mathcal{B}^{\alpha,A}_{t}=\phi_{t}(\mathcal{B}^{\alpha,A}_{1})\,.$$
Let us note that $\lambda$ scales under dilations according to the homogeneous dimension, not the Euclidean dimension, so in particular
$$\lambda(\mathcal{B}^{\alpha,A}_{t})=t^{2q+2}\lambda(\mathcal{B}^{\alpha,A}_{1})\,.$$
Thus, in the set-up of \eqref{eq:1.2}, the lattice point counting problem is that of establishing an asymptotic estimate of the form
\begin{equation}\label{eq:1.6}
|\mathbb{Z}^{2q+1}\cap\mathcal{B}^{\alpha,A}_{t}|=\lambda(\mathcal{B}^{\alpha,A}_{t})+\mathcal{O}_{\epsilon}(\lambda(\mathcal{B}^{\alpha,A}_{t}))^{\kappa+\epsilon}
\end{equation}
for any $\epsilon>0$, where $\kappa$ is taken to be the infimum over all admissible exponents. In \eqref{eq:1.6}, the involved parameters are assumed to be fixed so we do not display their dependence in the error term.\\\\
In view of \eqref{eq:1.6},  the counting problem on $\mathcal{H}_{q}$ is that of counting integral points in the Euclidean lattice $\mathbb{Z}^{2q+1}$ contained in the family of increasing bodies $\phi_{t}(\mathcal{B}^{\alpha,A}_{1})=\mathcal{B}^{\alpha,A}_{t}\subset\mathbb{R}^{2q+1}$
as $t\to\infty$. Note that the Heisenberg dilations used to expand $\mathcal{B}^{\alpha,A}_{1}$ are materially different from the Euclidean dilations. Also, one needs to obtain a good understanding of the geometric properties of these bodies, as they play a crucial role in the analysis of the error term. For example, it turns out that $\mathcal{B}^{\alpha,A}_{1}$ is Euclidean convex if and only if $\alpha\geq 2$, and its boundary surface contains points of vanishing  Gaussian curvature for values of $\alpha\geq 2$. For further details, we refer the reader to \cite{garg2015lattice}.
\subsection{Statement of results}
We shall now confine ourselves to the case of $\mathcal{H}_{1}$ where the norm will be taken to be the Cygan-Kor$\acute{a}$nyi Heisenberg-norm $\mathcal{N}_{4,1}$. In \cite{garg2015lattice}, the upper-bound $\kappa\leq\frac{1}{2}$ has been established,  and we shall prove the opposite inequality $\kappa\geq\frac{1}{2}$, thereby establishing \eqref{eq:1.6} with the value of $\kappa=\frac{1}{2}$. We begin by introducing some notations.\\\\ 
Given a parameter $t\geq1$, we have the following equality $$|\mathbb{Z}^{3}\cap\mathcal{B}^{4,1}_{t}|=\sum_{0\leq m^2+n^2<\, t^{4}}r_2(m)\,,$$
where $r_2(m)=|\{a,b\in\mathbb{Z}:\,a^{2}+b^{2}=m\}|$.\\\\
Setting $x=t^{4}$, the counting problem is that of estimating the sum
\begin{equation}
\mathcal{S}(x)=\sum_{0\leq m^2+n^2<\, x}r_2(m)\,.
\end{equation}
Since $\mathcal{S}(x)\sim \frac{\pi^{2}}{2}x$ (to see this, first execute the summation over the variable $n$, and then proceed by applying partial summation), the error term takes the following form
\begin{equation}\label{eq:1.8}
\mathcal{E}(x)=\mathcal{S}(x)-\frac{\pi^{2}}{2}x.
\end{equation}
We are now ready to state our main results.
\begin{thm}
With $\mathcal{E}(x)$ as above, we have:
\begin{equation}
\limsup_{x\to\infty}\frac{\mathcal{E}(x)}{x^{\frac{1}{2}}}=\infty
\end{equation}
\begin{equation}
\liminf_{x\to\infty}\frac{\mathcal{E}(x)}{x^{\frac{1}{2}}}=-\infty
\end{equation}
\end{thm}
In \cite{garg2015lattice}, the upper-bound $\left|\mathcal{E}(x)\right|\ll x^{\frac{1}{2}}\log{x}$ has been established, hence as an immediate corollary we deduce: 
\begin{cor}
The exponent $\frac{1}{2}$ in the upper-bound $\left|\mathcal{E}(x)\right|\ll x^{\frac{1}{2}}\log{x}$ obtained by Garg, Nevo \& Taylor can not be improved and is thus best possible, thereby establishing \eqref{eq:1.6} with the value of $\kappa=\frac{1}{2}$.
\end{cor}
\textit{\textbf{On the method of proof}}. One may approach the sum $\mathcal{S}(x)$ in several ways, one of which is by first replacing the sharp cutoff function $\chi_{[0,x)}$ by an appropriately chosen smooth function of compact support,  and then apply Poisson summation for the smoothed sum. The zero frequency gives rise to the main term, and one is left with estimating the tail. As we wish to retain the true order of magnitude of $\mathcal{E}(x)$, we opt to choose a different approach as we now explain.\\\\
First we note that $2\mathcal{M}(x)$ differs from $\mathcal{S}(x)$ by an amount of size at most $\mathcal{O}(x^{\frac{1}{2}})$, where 
\begin{equation}\label{eq:1.11}
\mathcal{M}(x)=\sum_{0\leq m\leq\, \sqrt[]{x}}r_2(m)(x-m^{2})^{\frac{1}{2}}\,.
\end{equation}
Thus, our theorem will follow if we can prove that the error term arising from the approximation of $\mathcal{M}(x)$ obeys the two sided estimate $\Omega_{\pm}(x^{\frac{1}{2}})$.\\\\
The presence of $(x-m^{2})^{\frac{1}{2}}$ in the above sum is more than welcome, as it serves as a smoothing factor. Already at this stage, by applying standard methods of contour integration one can establish the bound $\left|\mathcal{E}(x)\right|\ll x^{\frac{1}{2}}(\log{x})^{C}$ for some constant $C>0$. By performing a more careful analysis, one could obtain a closed form expression for the error term arising in the estimation of \eqref{eq:1.11}, however it turns out to be difficult to extract any information as to its true order of magnitude. The reason behind this difficulty is that the variable of summation $m$ appears in $(x-m^{2})^{\frac{1}{2}}$ with a squared factor. To see this more clearly, by Mellin inversion with $\sigma>1$, we have
\begin{equation}\label{eq:1.12}
M(x)\,=\,\sqrt[]{x}+ \frac{c}{2\pi i}\int_{(\sigma)}^{}\frac{\Gamma(\frac{s}{2})}{\Gamma(\frac{s+3}{2})} x^{\frac{s+1}{2}}Z(s)ds
\end{equation}
for some constant $c$, where
$$Z(s)=\sum_{m=1}^{\infty}\frac{r_2(m)}{m^{s}}.$$
The Gamma factor in \eqref{eq:1.12} does not match the one appearing in the functional equation for $Z(s)$, and this will cause some complications later on. To make the Gamma factors match, we first use Lemma 1 stated below to remove this square factor, and only then do we proceed to the evaluation of \eqref{eq:1.12}, leading us to the following sum
\begin{equation}\label{eq:1.13}
\mathcal{S}(\,\sqrt[]{x}\,;n)=\frac{(-1)^{n}}{n!}f^{(n)}(2x^{\frac{1}{2}})\sum_{0\leq m\leq\,\sqrt[]{x}}r_2(m)(\sqrt[]{x}-m)^{n+\frac{1}{2}}
\end{equation}
where $f(t)=t^{\frac{1}{2}}$, and $n$ runs through nonnegative integers .\\\\
The sum appearing in \eqref{eq:1.13} will be handled by Lemma 2 stated below, and we shall estimate it within an error term of size $\mathcal{O}(2^{-n}x^{\frac{1}{4}})$ uniformly in $n\geq 1$. The error term arising in the approximation of $\mathcal{S}(\,\sqrt[]{x}\,;0)$ will obey the estimate $\Omega_{\pm}(x^{\frac{1}{2}})$ which is exactly what we need. Having Lemmas 1 and 2 at our disposal, we proceed to the proof of Theorem 1.
\section{Preparatory Lemmas}
\begin{lem}
Let $f(t)=t^{\frac{1}{2}}$, and fix some $x\geq 1$. Then for $0\leq m\leq\,\sqrt[]{x}$ a non negative integer we have:
$$(x-m^{2})^{\frac{1}{2}}=\sum_{n=0}^{\infty}\frac{(-1)^{n}}{n!}f^{(n)}(2x^{\frac{1}{2}})(x^{\frac{1}{2}}-m)^{n+\frac{1}{2}}\,.$$
\end{lem}
\begin{proof}
Let $0\leq m\leq\,\sqrt[]{x}$ be a non negative integer. We first note that
$$(x-m^{2})^{\frac{1}{2}}=f(2x^{\frac{1}{2}})f(x^{\frac{1}{2}}-m)+f(x^{\frac{1}{2}}-m)(f(x^{\frac{1}{2}}+m)-f(2x^{\frac{1}{2}}))\,.$$
Next, we expand the function $f$ around $2x^{\frac{1}{2}}$ getting
$$f(x^{\frac{1}{2}}+m)-f(2x^{\frac{1}{2}})=\sum_{n=1}^{\infty}\frac{(-1)^{n}}{n!}f^{(n)}(2x^{\frac{1}{2}})(x^{\frac{1}{2}}-m)^{n}\,.$$
Inserting this into the above equation gives
\[(x-m^{2})^{\frac{1}{2}}=\sum_{n=0}^{\infty}\frac{(-1)^{n}}{n!}f^{(n)}(2x^{\frac{1}{2}})(x^{\frac{1}{2}}-m)^{n+\frac{1}{2}}\]
as claimed.
\end{proof}
\begin{lem}
Let $x\in \mathbb{R}_{+}$ be large. For $n\in\mathbb{N}_{\geq 0}$ define:
$$\mathcal{S}(\,\sqrt[]{x}\,;n)=\frac{(-1)^{n}}{n!}f^{(n)}(2x^{\frac{1}{2}})\sum_{0\leq m\leq\,\sqrt[]{x}}r_2(m)(\sqrt[]{x}-m)^{n+\frac{1}{2}}\,.$$
Then for $n\geq 1$ we have
\begin{equation}\label{eq:2.1}
\mathcal{S}(\,\sqrt[]{x}\,;n)=c_nx +\mathcal{O}(2^{-n}x^{\frac{1}{4}})
\end{equation}
where the coefficients $c_n$ are given by
\begin{equation}\label{eq:2.2}
c_n=-\pi\frac{\prod_{k=1}^{n-1}(1-\frac{1}{2k})}{\sqrt[]{2}\,2^{n}n(n+\frac{3}{2})}
\end{equation}
and the implied constant is absolute.\\\\
For $n=0$ we have
\begin{equation}\label{eq:2.3}
\mathcal{S}(\,\sqrt[]{x}\,;0)=c_0x + x^{\frac{1}{2}}Q(\,\sqrt[]{x}\,)
\end{equation}
with $c_0=\frac{2^{\frac{3}{2}}\pi}{3}$. The second term on the RHS of \eqref{eq:2.3} satisfies
\begin{equation}\label{eq:2.4}
\limsup_{x\to\infty}Q(\,\sqrt[]{x}\,)=\infty,\,\,\,\liminf_{x\to\infty}Q(\,\sqrt[]{x}\,)=-\infty
\end{equation}
and
\begin{equation}\label{eq:2.5}
\big|Q(\,\sqrt[]{x}\,)\big|\ll x^{\frac{1}{4}}
\end{equation}
\end{lem}
\begin{proof}
We fix some $n\in\mathbb{N}_{\geq 0}$, and split the proof into two cases as follows:\\\\
\textit{\textbf{Case 1}}: $n\geq 2$.\\\\
We first begin by writing
\begin{equation}\label{eq:2.6}
\sum_{0\leq m\leq\,\sqrt[]{x}}r_2(m)(\sqrt[]{x}-m)^{n+\frac{1}{2}}\,=\,x^{\frac{n}{2}+\frac{1}{4}}+\sum_{m=1}^{\infty}r_2(m)g(m)
\end{equation}
where 
\[g(y)=\left\{
        \begin{array}{ll}
            (\sqrt[]{x}-y)^{n+\frac{1}{2}}& ; \,0< y<\, \sqrt[]{x} \\\\
            0 & ;\, otherwise
        \end{array}
    \right.\]
\\Of course $g(y)$ depends on $n$, but we do not display it.
\\\\The Mellin transform of $g$, valid for $\Re(s)>0$, is given by
$$\int_{0}^{\infty}g(y)y^{s-1}dy=x^{\frac{n}{2}+\frac{1}{4}+\frac{s}{2}}\int_{0}^{1}(1-y)^{n+\frac{1}{2}}y^{s-1}dy=x^{\frac{n}{2}+\frac{1}{4}}G(s)x^{\frac{s}{2}}$$
where 
$$G(s)=\frac{\Gamma(s) \Gamma(n+\frac{3}{2})}{\Gamma(s+n+\frac{3}{2})}$$
\\with $\Gamma(s)$ being the familiar Gamma function.
\\\\For later use, let us record the following estimate for the Gamma function (see \cite{iwaniec2004analytic} 5.A (5.113))
$$\Gamma(u+it)=\sqrt[]{2\pi}(it)^{u-\frac{1}{2}}e^{-\frac{\pi}{2}|t|}\bigg(\frac{|t|}{e}\bigg)^{it}\bigg(1+\mathcal{O}\bigg(\frac{1}{|t|}\bigg) \bigg)$$
\\wich is valid uniformly for $\alpha<u<\beta$ with any fixed $\alpha,\beta\in\mathbb{R}$, provided $|t|$ is large enough in terms of $\alpha$ and $\beta$.  
\\\\By Mellin inversion, we get the following integral representation valid for $y\in\mathbb{R}_{+}$ 
$$g(y)=\frac{x^{\frac{n}{2}+\frac{1}{4}}}{2\pi i}\int_{(\sigma)}^{}G(s)x^{\frac{s}{2}}y^{-s}ds$$
with, say, $\sigma=1+\frac{1}{\log{x}}$. Note that the estimate for $\Gamma(s)$ implies the absolute convergence of the above integral.\\\\
Substituting $y=m$ into this last equation, multiplying by $r_{2}(m)$ and then summing over all $m\in\mathbb{N}$, we arrive at
\begin{equation}\label{eq:2.7}
\sum_{m=1}^{\infty}r_2(m)g(m)=\frac{x^{\frac{n}{2}+\frac{1}{4}}}{2\pi i}\int_{(\sigma)}^{}G(s)x^{\frac{s}{2}}Z(s)ds
\end{equation}
where
$$Z(s)=\sum_{m=1}^{\infty}\frac{r_2(m)}{m^{s}}\,.$$
Here, we have interchanged the order of summation and integration which is justified since $\sigma>1$.\\\\
The Zeta function $Z(s)$, initially defined for $\Re(s)>1$, admits an analytic continuation to the entire complex plane, except at $s=1$ where it has a simple pole with residue $\pi$, and satisfies the following functional equation (see \cite{iwaniec2004analytic} 5.10)
\begin{equation}\label{eq:2.8}
\pi^{-s}\Gamma(s)Z(s)=\pi^{-(1-s)}\Gamma(1-s)Z(1-s)\,.
\end{equation}
In the strip $-\frac{1}{2}\leq \Re(s)\leq\sigma$, we have the following convexity bound due to H.Rademacher (see  \cite{rademacher1959phragmen} Theorem 4 (7.4))
\begin{equation}\label{eq:2.9}
|Z(s)|\ll|s|^{2}\,,\,\text{as}\,\,|\Im(s)|\to\infty\,.
\end{equation}
\\Now we are ready to evaluate \eqref{eq:2.7}. Write $\mathcal{C}$ for the positively oriented rectangle contour with vertices $\{-\frac{1}{2}\pm i\infty,\,\sigma\pm i\infty\}$.\\\\
$G(s)x^{\frac{s}{2}}Z(s)$ has two simple poles in the interior of $\mathcal{C}$ located at $s=0,1$. By \eqref{eq:2.8} we compute:
$$Res_{s=1}G(s)x^{\frac{s}{2}}Z(s)=\frac{\Gamma(1) \Gamma(n+\frac{3}{2})}{\Gamma(1+n+\frac{3}{2})}x^{\frac{1}{2}}\lim_{s\to 1}(s-1)Z(s)=\frac{\pi}{n+\frac{3}{2}}x^{\frac{1}{2}}$$
$$Res_{s=0}G(s)x^{\frac{s}{2}}Z(s)=\lim_{s\to 0}s\Gamma(s)Z(s)=-1$$
\\By the convexity bound \eqref{eq:2.9}, together with the estimate given for the Gamma function, it follows that
$$|G(s)x^{\frac{s}{2}}Z(s)|\to\ 0\,,\,\text{as}\,\,|\Im(s)|\to\infty$$
in the strip $-\frac{1}{2}\leq \Re(s)\leq\sigma$.
\\\\Hence, by moving the line of integration to $\Re(s)=-\frac{1}{2}$, we have by the theorem of residues
\begin{equation}\label{eq:2.10}
\frac{1}{2\pi i}\int_{(\sigma)}^{}G(s)x^{\frac{s}{2}}Z(s)ds=-1+\frac{\pi}{n+\frac{3}{2}}x^{\frac{1}{2}}+\frac{1}{2\pi i}\int_{(-\frac{1}{2})}^{}G(s)x^{\frac{s}{2}}Z(s)ds\,.
\end{equation}
\\\\Multiplying \eqref{eq:2.10} by $x^{\frac{n}{2}+\frac{1}{4}}$, \eqref{eq:2.7} takes the form
\begin{equation}\label{eq:2.11}
\sum_{m=1}^{\infty}r_2(m)g(m)=\,-x^{\frac{n}{2}+\frac{1}{4}}+\frac{\pi}{n+\frac{3}{2}}x^{\frac{n}{2}+\frac{3}{4}} +\frac{x^{\frac{n}{2}+\frac{1}{4}}}{2\pi i}\int_{(-\frac{1}{2})}^{}G(s)x^{\frac{s}{2}}Z(s)ds\,.
\end{equation}
\\\\Inserting \eqref{eq:2.11} back into \eqref{eq:2.6}, we derive 
\begin{equation}\label{eq:2.12}
\sum_{0\leq m\leq\,\sqrt[]{x}}r_2(m)(\sqrt[]{x}-m)^{n+\frac{1}{2}}\,=\,\frac{\pi}{n+\frac{3}{2}}x^{\frac{n}{2}+\frac{3}{4}}+\frac{x^{\frac{n}{2}+\frac{1}{4}}}{2\pi i}\int_{(-\frac{1}{2})}^{}G(s)x^{\frac{s}{2}}Z(s)ds\,.
\end{equation}
Now
$$\frac{(-1)^{n}}{n!}f^{(n)}(2x^{\frac{1}{2}})=-\frac{\prod_{k=1}^{n-1}(1-\frac{1}{2k})}{\sqrt[]{2}\,2^{n}n}x^{\frac{1}{4}-\frac{n}{2}}$$
and by setting
\[c_n=-\pi\frac{\prod_{k=1}^{n-1}(1-\frac{1}{2k})}{\sqrt[]{2}\,2^{n}n(n+\frac{3}{2})}\]
\\we arrive at the following identity (after multiplying \eqref{eq:2.12} by $\frac{(-1)^{n}}{n!}f^{(n)}(2x^{\frac{1}{2}})$)
\begin{equation}\label{eq:2.13}
\mathcal{S}(\,\sqrt[]{x}\,;n)=c_nx + \frac{C_nx^{\frac{1}{2}}}{2\pi i}\int_{(-\frac{1}{2})}^{}G(s)x^{\frac{s}{2}}Z(s)ds\,
\end{equation}
where
\[C_n=-\frac{\prod_{k=1}^{n-1}(1-\frac{1}{2k})}{\sqrt[]{2}\,2^{n}n}\,.\]
\\It remains to estimate the second term on the RHS of \eqref{eq:2.13}.\\\\ 
By repeated use of the functional equation for the Gamma function, we have
$$G(s)=\frac{\Gamma(s) \Gamma(\frac{3}{2})}{\Gamma(s+\frac{3}{2})}\prod_{k=1}^{n}\frac{k+\frac{1}{2}}{s+k+\frac{1}{2}}=\frac{\Gamma(s) \Gamma(\frac{3}{2})}{\Gamma(s+\frac{3}{2})}\prod_{k=1}^{n}\frac{1}{1+\frac{s}{k+\frac{1}{2}}}$$
thus
$$|G(s)|=\bigg|\frac{\Gamma(s) \Gamma(\frac{3}{2})}{\Gamma(s+\frac{3}{2})}\bigg|\prod_{k=1}^{n}\frac{1}{|1+\frac{s}{k+\frac{1}{2}}|}\,.$$
Assume $\Re(s)=-\frac{1}{2}$. Isolating the first 2 terms appearing in the above product, and estimating the remaining ones trivially, we get
\begin{equation}\label{eq:2.14}
|G(s)|\leq\frac{\sqrt[]{\pi}}{2}\bigg(\prod_{k=3}^{n}\frac{1}{1-\frac{1}{2k+1}}\bigg)\bigg|\frac{\Gamma(s)}{\Gamma(s+\frac{3}{2})(1+\frac{2}{3}s)(1+\frac{2}{5}s)}\bigg|
\end{equation}
where the first product appearing above may be empty in case $n=2$.\\\\
By \eqref{eq:2.14}, together with the bound
$$\bigg|C_n\prod_{k=3}^{n}\frac{1}{1-\frac{1}{2k+1}}\bigg|\leq 2^{-n}$$
we conclude that
\begin{equation}\label{eq:2.15}
\bigg|\frac{C_n x^{\frac{1}{2}}}{2\pi i}\int_{(-\frac{1}{2})}^{}G(s)x^{\frac{s}{2}}Z(s)ds\bigg|\leq 2^{-n}x^{\frac{1}{4}}\int_{(-\frac{1}{2})}^{}\bigg|\frac{\Gamma(s)Z(s)}{\Gamma(s+\frac{3}{2})(1+\frac{2}{3}s)(1+\frac{2}{5}s)}\bigg|ds
\end{equation}
\\The integral appearing on the RHS of \eqref{eq:2.15} is independent of $n$ and converges absolutely by \eqref{eq:2.9}, hence case 1 is proved.\\\\
\textit{\textbf{Case 2}}: $n=0,1$.\\\\
To handle this case, we shall appeal to two estimates proved by K. Chandrasekharan and Raghavan Narasimhan. These estimates, which are a special case of a much greater theory, can be found in \cite{chandrasekharan1961hecke} and \cite{chandrasekharan1962functional}. In order to state their results, we begin by setting up some notations.\\\\
For $\varrho>0$ and $y>0$, let $E_{\varrho}(y)$ be defined by
\begin{equation}\label{eq:2.16}
\sum_{0\leq m\leq y}r_2(m)(y-m)^{\varrho}=\frac{\Gamma(\varrho+1)}{\Gamma(\varrho+2)}\pi y^{\varrho+1}+E_{\varrho}(y)\,.
\end{equation}
Chandrasekharan and Narasimhan have considered $\tilde{E_{\varrho}}(y)= E_{\varrho}(y)/\Gamma(\varrho+1)$ instead, but this will make no difference for us since we are interested in evaluating \eqref{eq:2.16} only for the values $\varrho=\frac{1}{2},\frac{3}{2}$. Thus, the dependency of the forthcoming estimates with respect to the variable $\varrho$ will not concern us. In order to be consistent with \cite{chandrasekharan1961hecke}, we shall state their results in terms of $\tilde{E_{\varrho}}(y)$. With the above notation, we have the following estimates (see \cite{chandrasekharan1961hecke} P.502 case (iii))\\\\
\begin{equation}\label{eq:2.17}
|\tilde{E_{\varrho}}(y)|\ll y^{\frac{1}{2}(\varrho+\frac{1}{2})}
\end{equation}
\\provided $\varrho>\frac{1}{2}$ where the implied constant may depend on $\varrho$, while for $\varrho\leq\frac{1}{2}$ we have
\begin{equation}\label{eq:2.18}
\limsup_{y\to\infty}y^{-\frac{1}{2}(\varrho+\frac{1}{2})}\tilde{E_{\varrho}}(y)=\infty,\,\,\,\liminf_{y\to\infty}y^{-\frac{1}{2}(\varrho+\frac{1}{2})}\tilde{E_{\varrho}}(y)=-\infty\,.
\end{equation}
\\ Now we specialize to our case in hand, namely $y=\,\sqrt[]{x}$ and $\varrho=n+\frac{1}{2}$ with $n=0,1$.
\\\\Writing
\begin{equation}\label{eq:2.19}
\sum_{0\leq m\leq\,\sqrt[]{x}}r_2(m)(\sqrt[]{x}-m)^{n+\frac{1}{2}}=\frac{\pi}{n+\frac{3}{2}}x^{\frac{n}{2}+\frac{3}{4}}+E_{n+\frac{1}{2}}(\,\sqrt[]{x\,}\,)
\end{equation}
\\we have by \eqref{eq:2.17} and \eqref{eq:2.18} the following estimates
\begin{equation}\label{eq:2.20}
|E_{\frac{3}{2}}(\,\sqrt[]{x\,}\,)|\ll\,x^{\frac{1}{2}}
\end{equation}
\begin{equation}\label{eq:2.21}
\limsup_{x\to\infty}x^{-\frac{1}{4}}E_{\frac{1}{2}}(\,\sqrt[]{x\,}\,)=\infty,\,\,\,\liminf_{x\to\infty}x^{-\frac{1}{4}}E_{\frac{1}{2}}(\,\sqrt[]{x\,}\,)=-\infty
\end{equation}
\\Note that by partial summation we have the easy bound
\begin{equation}\label{eq:2.22}
\big|E_{\frac{1}{2}}(\,\sqrt[]{x\,}\,)\big|\ll x^{\frac{1}{2}}\,.
\end{equation}
\\Multiplying \eqref{eq:2.19} by $\frac{(-1)^{n}}{n!}f^{(n)}(2x^{\frac{1}{2}})$, \eqref{eq:2.20} gives in the case of $n=1$
\begin{equation}
\mathcal{S}(\,\sqrt[]{x}\,;1)=c_1x +O(x^{\frac{1}{4}})
\end{equation}
\\ where $c_1=-\frac{\pi}{5\,\sqrt[]{2}}$, while for $n=0$ we have
\begin{equation}
\mathcal{S}(\,\sqrt[]{x}\,;0)=c_0x + \sqrt[]{2}x^{\frac{1}{4}}E_{\frac{1}{2}}(\,\sqrt[]{x\,}\,)=c_0x + x^{\frac{1}{2}}Q(\,\sqrt[]{x}\,)
\end{equation}
\\ where $c_0=\frac{2^{\frac{3}{2}}\pi}{3}$, and $Q(\,\sqrt[]{x}\,)=\,\sqrt[]{2}x^{-\frac{1}{4}}E_{\frac{1}{2}}(\,\sqrt[]{x\,}\,)$.\\\\
By \eqref{eq:2.21} and \eqref{eq:2.22} we have
\begin{equation}\label{eq:2.25}
\limsup_{x\to\infty}Q(\,\sqrt[]{x}\,)=\infty,\,\,\,\liminf_{x\to\infty}Q(\,\sqrt[]{x}\,)=-\infty
\end{equation}
and
\begin{equation}\label{eq:2.26}
|Q(\,\sqrt[]{x}\,)|\ll x^{\frac{1}{4}}.
\end{equation}
Thus case 2 is proved.
\end{proof}
\section{Proof of Theorem 1}
\begin{proof}[\textit{Proof} (Theorem 1)]
\begin{equation}\label{eq:3.1}
\mathcal{S}(x)=\sum_{0\leq m^2+n^2<\, x}r_2(m)=2\mathcal{M}(x)+\mathcal{O}(\,\sqrt[]{x}\,)
\end{equation}
where 
$$\mathcal{M}(x)=\sum_{0\leq m\leq\, \sqrt[]{x}}r_2(m)(x-m^{2})^{\frac{1}{2}}\,.$$
Using Lemma 1, and switching the order of summation, we have
\begin{equation}
\mathcal{M}(x)=\sum_{n=0}^{\infty}\mathcal{S}(\,\sqrt[]{x}\,;n)\,.
\end{equation}
Setting $$c=\sum_{n=0}^{\infty}c_n$$
where the sum clearly converges (see \eqref{eq:2.2}), we have by Lemma 2 \eqref{eq:2.1}, \eqref{eq:2.3}
\begin{equation}\label{eq:3.3}
\mathcal{M}(x)= cx + x^{\frac{1}{2}}Q(\,\sqrt[]{x}\,) +\mathcal{O}(x^{\frac{1}{4}})\,.
\end{equation}
Inserting \eqref{eq:3.3} back into \eqref{eq:3.1}, we derive
\begin{equation}\label{eq:3.4}
S(x)= 2cx + 2x^{\frac{1}{2}}Q(\,\sqrt[]{x}\,) +\mathcal{O}(x^{\frac{1}{2}})\,.
\end{equation}
By \eqref{eq:2.5}, the second term appearing on the RHS of (3.4) is bounded by $x^{\frac{3}{4}}$, and since $\mathcal{S}(x)\sim \frac{\pi^{2}}{2}x$ as $x\to\infty$, it follows that $2c=\frac{\pi^{2}}{2}$.\\
Subtracting $\frac{\pi^{2}}{2}x$ from both sides of \eqref{eq:3.4}, and dividing throughout by $x^{\frac{1}{2}}$, we arrive at
\begin{equation}
\frac{\mathcal{E}(x)}{x^{\frac{1}{2}}}=2Q(\,\sqrt[]{x}\,)+\mathcal{O}(1)\,.
\end{equation}
Taking $\limsup$ and $\liminf$, and using \eqref{eq:2.4}, concludes the proof.
\end{proof}
\textit{\textbf{Concluding remarks}}. Our results remain valid in the case where the norm is taken to be $\mathcal{N}_{4,A}$ for any $A>0$, as long as we consider this variable to be fixed. If we where to allow $A$ to vary with $x$, then the question of uniformity comes into play, and it would be an interesting problem to consider. Likewise, one may consider the lattice point counting problem when one or both of the variables of summation are restricted by congruence conditions to various moduli which are allowed to vary with $x$. It should be mentioned though, that the value of $\alpha=4$ was crucial to our arguments for obtaining the desired $\Omega_{\pm}$ estimates, while for different values we would have to settle for upper-bounds. 


Let us now say a few words about higher dimensional Heisenberg groups, where the norm is still the Cygan-Kor$\acute{a}$nyi Heisenberg-norm. As in the case of $\mathcal{H}_{1}$, the sum under consideration takes the form
$$\mathcal{S}_{q}(x)=\sum_{0\leq m^2+n^2<\, x}r_{2q}(m)=\sum_{l<\, x}\rho_{q}(l)$$
where 
$$\rho_{q}(l)=\sum_{m^2+n^2=l}r_{2q}(m)\,.$$
Now, for integers $l$ such that $r_{2}(l)=0$ we have $\rho_{q}(l)=0$ no matter what the value of $q$ is, and thus they do not make any contribution to the sum $\mathcal{S}_{q}(x)$, while those for which $r_{2}(l)\neq0$ will have a larger contribution as $q$ grows due to the presence of $r_{2q}(\cdotp)$ in $\rho_{q}(\cdotp)$. Hence, $\rho_{q}(\cdotp)$ viewed as a function of $q$ will exhibit higher spikes as $q$ increases, making the task of obtaining $\Omega_{\pm}$ estimates for the error term $\mathcal{E}_{q}(x)$ much harder. This is in sharp contrast to the case of counting lattice points in Euclidean balls, where the function $r_{k}(\cdotp)$ exhibits more of a regular  behavior as the dimension $k$ increases, and as soon as $k\geq 4$ one has a complete solution to the lattice point counting problem in this case.
\bibliographystyle{plain}
\bibliography{refs.bib}

%% file: main.bbl
\begin{thebibliography}{10}

\bibitem{chamizo1998lattice}
Fernando Chamizo.
\newblock Lattice points in bodies of revolution.
\newblock {\em Acta Arithmetica}, 85(3):265--277, 1998.

\bibitem{chandrasekharan1961hecke}
K~Chandrasekharan and Raghavan Narasimhan.
\newblock Hecke's functional equation and the average order of arithmetical
  functions.
\newblock {\em Acta arithmetica}, 6(4):487--503, 1961.

\bibitem{chandrasekharan1962functional}
K~Chandrasekharan and Raghavan Narasimhan.
\newblock Functional equations with multiple gamma factors and the average
  order of arithmetical functions.
\newblock {\em Annals of Mathematics}, pages 93--136, 1962.

\bibitem{cowling2010unitary}
Michael Cowling.
\newblock Unitary and uniformly bounded representations of some simple lie
  groups.
\newblock In {\em Harmonic Analysis and Group Representation}, pages 50--128.
  Springer, 2010.

\bibitem{cygan1978wiener}
Jacek Cygan.
\newblock Wiener's test for the brownian motion on the heisenberg group.
\newblock In {\em Colloquium Mathematicae}, volume~39, pages 367--373.
  Institute of Mathematics Polish Academy of Sciences, 1978.

\bibitem{cygan1981subadditivity}
Jacek Cygan.
\newblock Subadditivity of homogeneous norms on certain nilpotent lie groups.
\newblock {\em Proceedings of the American Mathematical Society}, 83(1):69--70,
  1981.

\bibitem{garg2015lattice}
Rahul Garg, Amos Nevo, and Krystal Taylor.
\newblock The lattice point counting problem on the heisenberg groups [le
  probl{\`e}me d{\'e}nombrement des points d’un r{\'e}seau dans les groupes
  de heisenberg].
\newblock In {\em Annales de l'institut Fourier}, volume~65, pages 2199--2233,
  2015.

\bibitem{gorodnik2012counting}
Alexander Gorodnik and Amos Nevo.
\newblock Counting lattice points.
\newblock {\em Journal f{\"u}r die reine und angewandte Mathematik (Crelles
  Journal)}, 2012(663):127--176, 2012.

\bibitem{heath1999lattice}
Roger Heath-{B}rown.
\newblock Lattice points in the sphere.
\newblock In {\em In-Number theory in progress}. Citeseer, 1999.

\bibitem{herz1962number}
Carl.~S Herz.
\newblock On the number of lattice points in a convex set.
\newblock {\em American Journal of Mathematics}, 84(1):126--133, 1962.

\bibitem{hlawka1950integrale}
Edmund Hlawka.
\newblock {\"U}ber integrale auf konvexen k{\"o}rpern i.
\newblock {\em Monatshefte f{\"u}r Mathematik}, 54(1):1--36, 1950.

\bibitem{huxley2003exponential}
Martin~N Huxley.
\newblock Exponential sums and lattice points iii.
\newblock {\em Proceedings of the London Mathematical Society}, 87(3):591--609,
  2003.

\bibitem{ivic2006lattice}
A~Ivi{\'c}, E~Kr{\"a}tzel, M~K{\"u}hleitner, and WG~Nowak.
\newblock Lattice points in large regions and related arithmetic functions:
  recent developments in a very classic topic (english summary), elementare und
  analytische zahlentheorie, 89--128, schr. wiss. ges. johann wolfgang goethe
  univ. frankfurt am main, 20, 2006.

\bibitem{iwaniec2004analytic}
Henryk Iwaniec and Emmanuel Kowalski.
\newblock {\em Analytic number theory}, volume~53.
\newblock American Mathematical Soc., 2004.

\bibitem{koranyi1985geometric}
Adam Kor{\'a}nyi.
\newblock Geometric properties of {H}eisenberg-type groups.
\newblock {\em Advances in Mathematics}, 56(1):28--38, 1985.

\bibitem{kratzel1999lattice}
Ekkehard Kr{\"a}tzel.
\newblock Lattice points in super spheres.
\newblock {\em Commentationes Mathematicae Universitatis Carolinae},
  40(2):373--392, 1999.

\bibitem{kratzel2002lattice1}
Ekkehard Kr{\"a}tzel.
\newblock Lattice points in some special three-dimensional convex bodies with
  points of {G}aussian curvature zero at the boundary.
\newblock {\em Comment. Math. Univ. Carolin}, 43(4):755--771, 2002.

\bibitem{kratzel2002lattice2}
Ekkehard Kr{\"a}tzel.
\newblock Lattice points in three-dimensional convex bodies with points of
  {G}aussian curvature zero at the boundary.
\newblock {\em Monatshefte f{\"u}r Mathematik}, 137(3):197--211, 2002.

\bibitem{atzel1988lattice}
Ekkehard Krätzel.
\newblock Lattice points, {K}luwer, {D}ordrecht {B}oston {L}ondon, 1988.

\bibitem{rademacher1959phragmen}
Hans Rademacher.
\newblock On the {P}hragm{\'e}n-{L}indel{\"o}f theorem and some applications.
\newblock {\em Mathematische Zeitschrift}, 72(1):192--204, 1959.

\bibitem{stein1999harmonic}
Elias~{M}. {S}tein.
\newblock Harmonic analysis real-variable methods, orthogonality, and
  oscillatory integrals.
\newblock {\em Bull. Amer. Math. Soc}, 36:505--521, 1999.

\end{thebibliography}
